\documentclass[a4paper,11pt]{article}

\usepackage{hyperref}
\usepackage{amsmath,amsfonts,amscd,amsthm}
\usepackage{xcolor}
\usepackage{url}
\usepackage{pdflscape}
\usepackage{graphicx, psfrag}
\usepackage[english]{babel}
\usepackage[latin1]{inputenc}
\usepackage{siunitx}
\usepackage[textsize=tiny,colorinlistoftodos]{todonotes}
\usepackage{multido}
\usepackage{longtable}
\usepackage{tabu}
\usepackage{float}
\usepackage{multirow}

\usepackage[T1]{fontenc}
\usepackage[bitstream-charter]{mathdesign}
\usepackage{rotating}
\usepackage{arydshln}
\usepackage{url}
\usepackage{algorithm}
\usepackage{algorithmicx}
\usepackage[noend]{algpseudocode}
\usepackage{textcomp}
\usepackage{enumerate}
\usepackage{enumitem}
\usepackage{tikz}
\usetikzlibrary{matrix,decorations.pathreplacing,backgrounds,positioning,calc,shapes.misc,mindmap}

\setlist{noitemsep,leftmargin=20pt,topsep=0pt,parsep=1pt,partopsep=0pt}
\theoremstyle{plain}
\newtheorem{theorem}{Theorem}[section]
\newtheorem{lemma}{Lemma}[section]
\newtheorem{corollary}{Corollary}[section]

\newtheorem{convention}{Convention}[section]

\theoremstyle{definition}
\newtheorem{definition}{Definition}[section]
\newtheorem{notation}{Notation}[section]
\newtheorem{example}{Example}[section]

\theoremstyle{remark}
\newtheorem{remark}{Remark}[section]
\newtheorem{characteristic}[theorem]{Algorithmic characteristic}

\newtheorem{conjecture}[theorem]{Conjecture}

\DeclareMathOperator{\sigSym}{{\mathfrak{s}}}

\DeclareMathOperator{\headSym}{lt}
\DeclareMathOperator{\indexSym}{ind}
\DeclareMathOperator{\lcm}{lcm}

\newcommand{\lcmm}[2]{\mathrm{lcm}\left(\hdp{#1}, \hdp{#2} \right)}

\newcommand{\comment}[1]{}

\newcommand{\field}{\mathcal{K}}
\newcommand{\ring}{\mathcal{R}}

\newcommand{\numberGenerators}{m}
\newcommand{\gen}[1]{{#1}_1,\ldots,{#1}_\numberGenerators}
\newcommand{\mgen}{\boldsymbol{e}_1,\ldots,\boldsymbol{e}_\numberGenerators}
\newcommand{\module}{\ring^\numberGenerators}
\newcommand{\mon}{\mathcal{M}}
\newcommand{\modmon}{\mathcal{N}}
\newcommand{\const}{\kappa}
\newcommand{\hdsyz}{\ensuremath{\mathcal{H}}}
\newcommand{\pairset}{\mathcal{P}}

\newcommand{\proj}[1]{\overline{#1}}
\newcommand{\mbasis}[1]{\boldsymbol{e}_{#1}}
\newcommand{\gbasissome}{\alpha}
\newcommand{\gbasis}[1]{\gbasissome_{#1}}
\newcommand{\sig}[1]{\sigSym\left({#1}\right)}
\newcommand{\syz}[1]{\mathrm{syz}\left({#1}\right)}
\newcommand{\origBasis}{\mathcal G}
\newcommand{\basis}{\ensuremath{\origBasis}}

\newcommand{\ind}[1]{\indexSym\left({#1}\right)}

\newcommand{\support}[1]{\mathrm{sup}\left(#1\right)}

\newcommand{\rleqff}{\trianglelefteq_{\mathrm{add}}}

\newcommand{\rleqsb}{\trianglelefteq_{\mathrm{rat}}}

\newcommand{\rleq}{\trianglelefteq}

\newcommand{\set}[1]{\left\{{#1}\right\}}
\newcommand{\setBuilder}[2]{\left\{{#1}\left|{#2}\right.\right\}}

\newcommand{\ideal}[1]{\left\langle{#1}\right\rangle}
\newcommand{\lc}[1]{\mathrm{lc}\left({#1}\right)}
\newcommand{\hd}[1]{\headSym\left({#1}\right)}
\newcommand{\hdp}[1]{\hd{\proj{#1}}}

\newcommand{\spair}[2]{\ensuremath{\mathrm{spair}\left({#1},{#2}\right)}}
\newcommand{\spoly}[2]{\ensuremath{\mathrm{spol}\left({#1},{#2}\right)}}

\newcommand{\spp}[1]{\ensuremath{\left(\sig{#1},\proj{#1}\right)}}

\newcommand{\potl}{\ensuremath{<_\textrm{pot}}}

\newcommand{\topl}{\ensuremath{<_\textrm{top}}}

\newcommand{\dpotl}{\ensuremath{<_\textrm{d-pot}}}

\newcommand{\dtopl}{\ensuremath{<_\textrm{d-top}}}

\newcommand{\hdpotl}{\ensuremath{<_\textrm{lt-pot}}}

\newcommand{\hdtopl}{\ensuremath{<_\textrm{lt-top}}}

\newcommand{\schl}{\hdpotl}

\newcommand{\updatesyzwitharg}{\ensuremath{\text{\bf
  UpdateSyz}\left(\basis,\hdsyz\right)}}
\newcommand{\updatesyz}{\ensuremath{\text{\bf
  UpdateSyz}}}
\newcommand{\rewritablenoarg}{\ensuremath{\text{\bf Rewritable}}}
\newcommand{\rewritablenoargnew}{\ensuremath{\text{\bf Rewritable'}}}

\newcommand{\notrewritable}[1]{\ensuremath{\text{\bf not\;
  Rewritable}\left({#1,\basis\cup\hdsyz,\rleq}\right)}}

\newcommand{\std}{{\bf STD}}
\newcommand{\rba}{{\bf RB}}

\newcommand{\gvw}{{\bf GVW}}
\newcommand{\ggv}{{\bf G2V}}

\newcommand{\gggv}{{\bf ImpG2V}}

\newcommand{\rwc}{{\bf RC}}
\newcommand{\prc}{{\bf PC}}
\newcommand{\chc}{{\bf CC}}

\newcommand{\ff}{\ensuremath{\text{\bf F5}}}

\newcommand{\singular}{\ensuremath{\text{\sc Singular}}}

\newcommand*{\defeq}{\mathrel{\vcenter{\baselineskip0.5ex \lineskiplimit0pt
                     \hbox{\scriptsize.}\hbox{\scriptsize.}}}%
                     =}

\newcommand{\sreduction}{\ensuremath{\sigSym}-reduction}
\newcommand{\sreductions}{\ensuremath{\sigSym}-reductions}
\newcommand{\sreduce}{\ensuremath{\sigSym}-reduce}
\newcommand{\sreduces}{\ensuremath{\sigSym}-reduces}
\newcommand{\sreduced}{\ensuremath{\sigSym}-reduced}
\newcommand{\sreducible}{\ensuremath{\sigSym}-reducible}
\newcommand{\sreducing}{\ensuremath{\sigSym}-reducing}

\newcommand{\sreducer}{\ensuremath{\sigSym}-reducer}
\newcommand{\sreducers}{\ensuremath{\sigSym}-reducers}

\newcommand{\grobner}{Gr\"obner}
\newcommand{\faugere}{Faug\`ere}

\newcommand{\hideAppendix}[1]{}\newcommand{\appRef}[1]{the online appendix of \cite{pracgb}}\newcommand{\AppRef}[1]{The online appendix of \cite{pracgb}}

\begin{document}
\title{Predicting zero reductions in \grobner{} basis computations}
\author{%
Christian Eder\footnote{The author was supported by the
  EXACTA grant (ANR-09-BLAN-0371-01) of the French National
  Research Agency.}\\
c/o Department of Mathematics \\
University of Kaiserslautern \\
67653 Kaiserslautern, Germany \\
\url{ederc@mathematik.uni-kl.de}
}
\maketitle

\begin{abstract}
Since Buchberger's initial algorithm for computing \grobner{} bases in 1965 many
attempts have been taken to detect zero reductions in advance. Buchberger's
Product and Chain criteria may be known the most, especially in the
installaton of Gebauer and M\"oller. 
A relatively new approach are signature-based criteria which were first used 
in \faugere{}'s \ff{} algorithm in 2002. For regular input sequences these
criteria are known to compute no zero reduction at all. 
In this paper we give a detailed discussion on zero reductions and the
corresponding syzygies.
We explain how the different methods to predict them compare to each other and show
advantages and drawbacks in theory and practice. 
With this a new insight into algebraic structures underlying \grobner{}
bases and their computations might be achieved.
\end{abstract}
\section{Introduction}
Since $1965$ \cite{bGroebner1965} \grobner{} bases are practically feasible.
Besides theoretical studies and generalizations, see, for example,
~\cite{gpSingularBook2007,krCCA12009,krCCA22005,mayr-ritscher-2011}, one of the main algorithmic
improvements is the prediction of useless data during the computations.
That means to use criteria to detect zero reductions in advance
~\cite{bGroebnerCriterion1979,buchberger2ndCriterion1985}.
Gebauer and M\"oller gave an optimal
implementation of Buchberger's criteria~\cite{gmInstallation1988}, but showed that
not all zero reductions are discarded. In 2002 Faug\`ere presented the
\ff{} algorithm~\cite{fF52002Corrected} which uses new criteria based on
so-called ``signatures''. For regular input sequences \ff{}
does not compute any zero reduction at all. Over the years optimized variants of
\ff{} were presented, see, for example,
~\cite{epF5C2009,apF5CritRevised2011,ggvGGV2010,gvwGVW2011,rs-2012,erF5SB2013}.
Still, not much is known about the connection between Buchberger's criteria and
those being based on signatures. Do they cover each other? How do they behave in
different situations? 
In~\cite{gashPhD2008,gashF5t2009} Gash gave only a small note on possible
combinations of both attempts in very specific cases. In~\cite{gerdtHashemiG2V}
Gerdt and Hashemi have considered the usage of Buchberger's criteria in the
\ggv{} algorithm, i.e. restricted to a signature-based algorithm using
an incremental structure computing the \grobner{} basis. Gao, Volny and Wang
have added corresponding step to the 2013 revision of \gvw{}
(see~\cite{gvwGVW2013}).

Here we present a detailed discussion on the connection between
different attempts to predict zero reductions in \grobner{} basis computations.
After a short introduction to our notation in which we also introduce a generic
signature-based algorithm denoted \rba{} that mirrors Buchberger's algorithm,
we explain in
Section~\ref{sec:detection} Buchberger's Product and
Chain criteria as well as the Syzygy and Rewritten criteria of \rba{}. We
show that in general signature-based criteria
predict more zero reductions, but with the drawback of introducing a more restricted
reduction process called ``\sreduction{}'' (see Definition~\ref{def:sreduce}). 
In Section~\ref{sec:improve-rewritten} we explain how the Rewritten criterion
can be improved. Following that we show that the Rewritten criterion includes the
Chain criterion and we give a vivid presentation of this fact in
Example~\ref{ex:rew-stronger-than-chain-crit}.
Section~\ref{sec:covers-prod-crit} discusses why the Product
criterion is, in general, not completely covered by signature-based criteria.
Proving that we can safely use the Product criterion in \rba{} enables us to
present an optimized variant which introduces this criterion without
overhead.
As we see in Section~\ref{sec:exp-results} situations where
syzygies coming from the Product criterion improve computations are very rare.
Even more, we give a conjecture that for a specific module
monomial order on the signatures all syzygies coming from the Product criterion
are already known in \rba{}. 

Dear reader, please note that in here is no intention to introduce a new, more
efficient variant of already known signature-based \grobner{} basis algorithms.
This paper should be understood as a first step in understanding the connections
between different attempts to minimize the number of useless data in \grobner{}
basis computations.
This might lead to a better exploitation of, until now, unused underlying
algebraic structures in order to improve computations even further.

\section{Notations}
In this section we introduce notations and basic terminology used in this
publication. Readers familiar with signature-based algorithms
might skip this section. We extend the notation introduced in~\cite{erF5SB2013}.

Let $\ring$ be a polynomial ring over a field $\field$. All polynomials
$f\in \ring$ can be uniquely written as a finite sum $f=\sum_{\const_v x^v\in
  \mon}\const_v x^v$ where $\const_v \in\field$, $x^v\defeq\prod_ix_i^{v_i}$ and $\mon$ is
minimal. The elements of $\mon$ are the \emph{terms} of $f$. The \emph{support of $f$}
is defined by $\support f:=\left\{ \textrm{terms in $f$} \right\}$. A
\emph{monomial} is a polynomial with exactly one term. A monomial with
a coefficient of 1 is \emph{monic}. Neither monomials nor terms of
polynomials are necessarily monic. We write $f\simeq g$ for $f,g\in \ring$
if there exists a non-zero $\const \in\field$ such that $f=\const g$.

Let $\module$ be a free $\ring$-module and let $\mgen$ be
the canonical basis of unit vectors in $\module$.
$\alpha\in \module$ can be uniquely written as a finite sum
$\alpha=\sum_{a\mbasis i\in \modmon}a\mbasis i$ where the $a$ are monomials
and $\modmon$ is minimal. The elements of $\modmon$ are the \emph{terms} of
$\alpha$. A \emph{module monomial} is an element of $\module$ with exactly
one term. A module monomial with a coefficient of 1 is
\emph{monic}. Neither module monomials nor terms of module elements
are necessarily monic. Let $\alpha\simeq\beta$ for
$\alpha,\beta\in \module$ if $\alpha=\const \beta$ for some non-zero $\const \in\field$.

Let $\leq$ denote two different orders -- one for $\ring$ and one for
$\module$:
The order for $\ring$ is a monomial order, which means that it is a
well-order on the set of monomials in $\ring$ such that $a\leq b$ implies
$ca\leq cb$ for all monomials $a,b,c\in \ring$.
The order for $\module$ is a module monomial order which means that it
is a well-order on the set
of module monomials in $\module$ such that $S \leq T$ implies $cS\leq cT$
for all module monomials $S,T\in \module$ and monomials $c\in \ring$.
We require the two orders to be \emph{compatible} in the sense that $a\leq b$ if
and only if $a\mbasis i\leq b\mbasis i$ for all monomials $a,b\in \ring$
and $i=1,\ldots,m$.
Consider a finite sequence of polynomials $\gen f \in \ring$ called the
\emph{input (polynomials)}. We call $\gen f$ a regular sequence if
$f_i$ is a non-zero-divisor on $\ring / \ideal{f_1,\ldots,f_{i-1}}$ for
$i=2,\ldots,m$.
We define the homomorphism
$\alpha\mapsto\proj\alpha$ from $\module$ to $\ring$ by
$\proj\alpha\defeq\sum_{i=1}^m\alpha_if_i.$ An element $\alpha\in \module$
with $\proj\alpha=0$ is called a \emph{syzygy}. The module of all syzygies of
$\gen f$ is denoted by $\syz{\gen f}$.

The following compatible module monomial orders are commonly used in
sig\-na\-ture-based \grobner{} basis algorithms.
\begin{definition}
\label{def:module-monomial-order}
Let $<$ be a monomial order on $\ring$ and let $a e_i, b e_j$ be two module
monomials in $\module$.
\begin{enumerate}
\item $a e_i \potl b e_j$ iff $i < j$ or $i = j$ and $a < b$.
\item $a e_i \topl b e_j$ iff $a < b$ or $a = b$ and $i < j$.
\end{enumerate}
These two orders can be combined by either a weighted degree or a weighted
leading monomial:
\begin{enumerate}[resume]
\item $a e_i \dpotl b e_j$ iff $\deg\left(\proj{a e_i}\right)
< \deg\left(\proj{b e_j}\right)$ or
$\deg\left(\proj{a e_i}\right) = \deg\left(\proj{b e_j}\right)$ and $a\mbasis i
\potl b \mbasis j$. Similar we define $a \mbasis i \dtopl b \mbasis j$.
\item $a e_i \hdpotl b e_j$ iff $\hdp{a e_i} < \hdp{b e_j}$ or
$\hdp{a e_i} = \hdp{b e_j}$ and $a \mbasis i \potl b \mbasis j$. Similar
we define $a\mbasis i \hdtopl b \mbasis j$.
\end{enumerate}
\end{definition}
Note that some of these orders are not efficient for signature-based
\grobner{} basis computations (see, for example,~\cite{gvwGVW2011}). Here we
concentrate on $\potl$, $\dpotl$ and $\schl$. Note that changing $i<j$ to
$-i<-j$ in the above definition of $\potl$ we receive the module order used
in~\cite{fF52002Corrected} for the \ff{} algorithm.

Next we introduce the notion of signatures. Note the connections and relations
to structures in the plain polynomial setting.

\begin{definition}\
\label{def:signature}
\begin{enumerate}
\item The \emph{lead term} $\hd f$ of $f\in
\ring\setminus\set 0$ is the $\leq$-maximal term of $f$. The \emph{lead
  coefficient} $\lc f$ of $f$ is the coefficient of $\hd f$. For a set $F
  \subset
  \ring$ we define the \emph{lead ideal of $F$} by $L(F) := \ideal{\hd f \mid f
    \in F}.$
\item The \emph{lead term} resp. \emph{signature} $\sig\alpha$ of
$\alpha\in \module\setminus\set 0$ denotes the
$\leq$-maximal term of $\alpha$. If $a\mbasis i=\sig\alpha$ then we call
$\ind{\alpha}\defeq i$ the \emph{index} of $\alpha$. For a set $M \in
  \module$ we define the \emph{lead module of $M$} by $L(M) := \ideal{\hd
    \alpha \mid \alpha
    \in M}.$

\item For $\alpha \in \module$ the \emph{sig-poly pair} of $\alpha$ is 
$\spp\alpha \in \module \times \ring$.
\item $\alpha,\beta\in \module$ are \emph{equal up to
  sig-poly pairs} if $\sig\alpha=\sig{\const \beta}$ and
$\proj\alpha=\proj{\const\beta}$ for some non-zero $\const\in\field$.
Correspondingly, $\alpha,\beta$ are said to be \emph{equal up to
  sig-lead pairs} if $\sig\alpha=\sig{\const\beta}$ and
$\hdp\alpha=\hdp{\const\beta}$ for some non-zero $\const \in\field$.
\end{enumerate}
\end{definition}

Next we introduce the notion of \grobner{} bases.
Let $f \in \ring$ and let $t$ be a term of~$f$. Then we
can \emph{reduce} $t$ by $g \in \ring$ if there exists a monomial $b$
such that $\hd{b g}=t$.
The outcome of the reduction step is $f-bg$ and
$g$ is called the \emph{reducer}.  When $g$ reduces $t$ we also say
that $bg$ reduces $f$. That way $b$ is introduced
implicitly instead of having to repeat the equation $\hd{bg}=t$.

The result of an reduction of $f \in \ring$ is an element $h \in \ring$ that has
been calculated from $f$ by a sequence of reduction steps. Thus, reductions can
always be assumed to be done w.r.t. some finite subset $G \subset \ring$.

Let $I= \langle \gen f \rangle$ be an ideal in $\ring$.
A finite subset $G$ of $\ring$ is a \emph{\grobner{} basis up to degree $d$}
for $I$ if $G \subset I$ and for all $f \in I$ with $\deg(f)\leq d$ $f$ reduces
to zero w.r.t. $G$. $G$ is a \emph{\grobner{} basis} for $I$ if $G$ is a
\grobner{} basis in all degrees.

In the very same way one can define \grobner{} basis with the notion of
standard representations: Let $f \in \ring$ and $G \subset \ring$ finite.
A representation $f = \sum_{i=1}^k m_i g_i$ with monomials $m_i \neq 0$, $g_i
\in G$ pairwise different is called a \emph{standard representation} if
$\max_\leq\left\{\hd{m_ig_i} \mid 1 \leq i \leq k\right\} \leq \hd f$. If for any
$f \in \ideal G$ with $f\neq 0$ $f$ has a standard representation w.r.t. $G$ and
$\leq$ then $G$ is a \grobner{} basis for $\ideal G$. Moreover, note that the
existence of a standard representation does not imply reducibility to zero, see,
for example, Exercise~5.63 in~\cite{bwkGroebnerBases1993}).

Moreover, Buchberger gave an algorithmic description of \grobner{} bases using
the notion of so-called S-polynomials:

Let $f \neq 0,g\neq 0 \in \ring$ and let $\lambda = \lcm\left(\hd f,\hd g\right)$ be the monic
least common multiple of $\hd f$ and $\hd g$. The
\emph{S-polynomial} between $f$ and $g$ is given by
\[\spoly{f}{g} \defeq \frac{\lambda}{\hd f} f - \frac{\lambda}{\hd g}g.\]

\begin{theorem}[Buchberger's criterion]
\label{thm:gb}
Let $I= \langle \gen f \rangle$ be an ideal in $\ring$.
A finite subset $G$ of $\ring$ is a \emph{\grobner{} basis}
for $I$ if $G \subset I$ and for all $f,g \in G$ $\spoly f g$ reduces
to zero w.r.t. $G$.
\end{theorem}

Every non-syzygy module element $\alpha\in \module$ has two main
associated characteristics -- the signature $\sig\alpha\in \module$ and
the lead term $\hdp\alpha\in \ring$ of its image $\proj\alpha$.
Lead terms and signatures include a coefficient for mathematical
convenience, though an implementation of a signature-based \grobner{} basis
algorithm working in polynomial rings over fields need not store the signature
coefficients.

In order to keep track of the signatures when reducing corresponding polynomial
data we get a classic polynomial reduction together with a further condition.

\begin{definition}
\label{def:sreduce}
Let $\alpha\in \module$ and let $t$ be a term of~$\proj\alpha$. Then we
can \emph{\sreduce{}} $t$ by $\beta\in \module$ if
\begin{enumerate}
\item\label{cond:preduce} there exists a monomial $b$ such that $\hdp{b\beta}=t$ and
\item\label{cond:sreduce} $\sig{b\beta}\leq\sig\alpha$.
\end{enumerate}
\end{definition}
The outcome of the \sreduction{} step is then $\alpha-b\beta$ and
$\beta$ is called the \emph{\sreducer{}}.  When $\beta$ \sreduces{} $t$ we also say for convenience
that $b\beta$ \sreduces{} $\alpha$. That way $b$ is introduced
implicitly instead of having to repeat the equation $\hdp{b\beta}=t$.

\begin{remark}
Note that Condition~(\ref{cond:preduce}) from Definition~\ref{def:sreduce} defines
a classic polynomial reduction step. It implies that $\hdp{b\beta} \leq \hdp{\alpha}$.
Moreover, Condition~(\ref{cond:sreduce}) lifts the above implication to $\module$ so that it
involves signatures. Since we are interested in computing \grobner{} bases in
$\ring$ one
can interpret an \sreduction{} of $\alpha$ by $\beta$ as classic polynomial
reduction of $\proj\alpha$ by $\proj\beta$ together with
Condition~(\ref{cond:sreduce}). Thus an \sreduction{} represents a connection
between data in $\ring$ and corresponding data in $\module$ when a polynomial reduction
takes place.
\end{remark}

Just as for classic polynomial reduction, if
$\hdp{b\beta}\simeq\hdp\alpha$ then the \sreduction{} step is a
\emph{top \sreduction{} step} and otherwise it is a \emph{tail
\sreduction{} step}. Analogously we define the distinction for
signatures: If $\sig{b\beta}\simeq\sig\alpha$ then the reduction
step is a \emph{singular \sreduction{} step} and otherwise it is a
\emph{regular \sreduction{} step}.

The result of an \sreduction{} of $\alpha\in \module$ is $\gamma\in \module$
that has been calculated from $\alpha$ through a sequence of
\sreduction{} steps such that $\gamma$ cannot be further
\sreduced{}. The reduction is a \emph{tail \sreduction} if only tail
\sreduction{} steps are allowed and it is a \emph{top \sreduction{}}
if only top \sreduction{} steps are allowed. The reduction is a
\emph{regular \sreduction} if only regular \sreduction{} steps are
allowed. $\alpha\in \module$ is \emph{\sreducible{}} if
it can be \sreduced{}.

If $\alpha$ \sreduces{} to $\gamma$ and
$\gamma$ is a syzygy then we say that $\alpha$ \emph{\sreduces{} to
zero} even if $\gamma\neq 0$.

Note that analogously to the classic polynomial reduction \sreduction{}
is always with respect to a finite \emph{basis} $\basis\subset
\module$. The \sreducers{} in \sreduction{} are chosen from the basis
$\basis$.

\begin{definition}
\label{def:sgb}
Let $I= \langle \gen f \rangle$ be an ideal in $\ring$.
A finite subset $\basis{} \subset \module$ is a \emph{signature \grobner{} basis in signature $T$} for $I$ if
all $\alpha\in \module$ with $\sig\alpha=T$ \sreduce{} to zero w.r.t.
\basis{}.
\basis{}
is a \emph{signature \grobner{} basis up to signature $T$} for $I$ if \basis{}
is a signature \grobner{} basis in all signatures $S$ such that
$S<T$.
\basis{} is a \emph{signature \grobner{} basis} for $I$ if it is a
signature \grobner{} basis for $I$ in all signatures. We denote $\proj\basis :=
\left\{\proj\alpha \mid \alpha \in \basis\right\} \subset \ring$.
\end{definition}

\begin{lemma}
Let $I= \langle \gen f \rangle$ be an ideal in $\ring$.
If \basis{} is a signature \grobner{} basis for $I$ then
$\proj\basis$ is a \grobner{} basis for $I$.
\end{lemma}

\begin{proof}
For example, see Section~$2.2$ in \cite{rs-2012}.
\end{proof}

\begin{convention}
In the following, when denoting $\basis{}\subset\module$ ``a signature
\grobner{} basis (up to signature $T$)'' we always mean ``a signature
\grobner{} basis (up to signature $T$) for $I=\langle \gen f \rangle$''. We omit
the explicit notion of the input ideal whenever it is clear from the context.
The same holds for classic \grobner{} bases $G \subset \ring$.
\end{convention}

As in the classic polynomial setting we want to give an algorithmic description
of signature \grobner{} bases. For this we introduce the notion of S-pairs.

\begin{definition}\
\begin{enumerate}
\item Let $\alpha,\beta \in \module$ such that $\proj\alpha \neq 0$, $\proj\beta
\neq 0$ and let the monic least common multiple
of $\hdp\alpha$ and $\hdp\beta$ be
$\lambda=\lcm\left(\hdp\alpha,\hdp\beta\right)$. The \emph{S-pair} between $\alpha$ and
$\beta$ is given by
\[\spair\alpha\beta \defeq \frac{\lambda}{\hdp\alpha}\alpha -
\frac{\lambda}{\hdp\beta}\beta.\]
\item $\spair\alpha\beta$ is \emph{singular} if
$\sig{\frac{\lambda}{\hdp\alpha}\alpha} \simeq \sig{\frac{\lambda}{\hdp\beta}\beta}$.
Otherwise it is \emph{regular}.
\end{enumerate}
\end{definition}

Note that $\spair\alpha\beta \in \module$ and $\proj{\spair\alpha\beta} =
\spoly{\proj\alpha}{\proj\beta}$.

\begin{theorem}
\label{thm:spairs}
Let $T$ be a module monomial of $\module$ and let $\basis{}\subset \module$ be a
finite basis. Assume that all regular S-pairs $\spair\alpha\beta$ with
$\alpha,\beta\in\basis$ and $\sig{\spair\alpha\beta} < T$ \sreduce{} to zero and all
$\mbasis i$ with $\mbasis i<T$ \sreduce{} to zero. Then $\basis$ is a
signature \grobner{} basis up to signature $T$.
\end{theorem}

\begin{proof}
For example, see Theorem~2 in \cite{rs-ext-2012}.
\end{proof}

Note the similarity of Theorem~\ref{thm:spairs} and Theorem~\ref{thm:gb}.
The outcome of classic polynomial reduction depends on the choice of
reducer which can change what the intermediate
bases are in the classic Buchberger algorithm. Lemma \ref{lem:p3}
implies that all S-pairs with the same signature yield the same
regular \sreduced{} result as long as we process S-pairs in order of
increasing signature.

\begin{lemma}
\label{lem:p3}
Let $\alpha,\beta\in \module$ and let \basis{} be a signature \grobner{}
basis up to signature $\sig\alpha=\sig\beta$. If $\alpha$ and $\beta$
are both regular top \sreduced{} then $\hdp\alpha=\hdp\beta$ or
$\proj\alpha=\proj\beta=0$. Moreover, if $\alpha$ and $\beta$ are both regular
\sreduced{} then $\proj\alpha=\proj\beta$.
\end{lemma}

\begin{proof}
For example, see Lemma~3 in \cite{rs-ext-2012}.
\end{proof}

We simplify notations using facts from the previous statements.

\begin{notation}\
\label{conv:spairs}
\begin{enumerate}
\item Due to Lemma~\ref{lem:p3} we assume in the following that $\basis$ always
denotes a finite subset of $\module$ with the
property that for $\alpha, \beta \in
\basis$ with $\sig\alpha\simeq\sig\beta$ it follows that
$\alpha=\beta$.
\label{conv:spairs:basis}
\item Theorem~\ref{thm:spairs} suggests to consider only regular S-pairs for the
computation of signature \grobner{} bases. Thus in the following ``S-pair''
always refers to ``regular S-pair''.\label{conv:spairs:item}
\end{enumerate}
\end{notation}
\section{Detecting zero reductions}
\label{sec:detection}
Different criteria to predict useless data during
the computation of a \grobner{} basis exist.
One such attempt goes back to Buchberger and implements two criteria.

\begin{lemma}[Product criterion~\cite{bGroebner1965,bGroebnerCriterion1979}]
Let $f,g \in \ring$ with $\lcm\left(\hd f,\hd g\right) = \hd f \hd g$.
Then $\spoly f g$ reduces to zero w.r.t.
$\left\{f,g\right\}$.
\label{lem:prod-crit}
\end{lemma}
In the above situation we also say that the S-polynomial \emph{$\spoly f g$ fulfills the
Product criterion}.

\begin{lemma}[Chain criterion~\cite{kollBuchberger1978,bGroebnerCriterion1979}]
Let $f,g,h \in \ring$, $G \subset \ring$ finite. If $\hd h \mid \lcm\left(\hd f,
\hd g\right)$, and if $\spoly f h$ and $\spoly h g$ have a standard representation
w.r.t. $G$ resp., then $\spoly f g$ has a standard representation w.r.t. $G$.
\label{lem:chain-crit}
\end{lemma}
In the above situation we also say that the triple \emph{$(f,g,h) \in \ring^3$ fulfills the
Chain criterion}.
Another formulation of Lemma~\ref{lem:chain-crit} (also explaining its name) is the relation
\begin{equation}
\spoly f g = u\; \spoly f h + v\; \spoly h g
\label{eq:chain-crit}
\end{equation}
for suitable $u,v \in \mon$. Clearly, if we know the standard representations of
$\spoly f h$ and $\spoly h g$, respectively, $\spoly f g$ does not need to be
considered when computing $G$. Depending on $u$ and $v$ we may even be free to
remove any of the three above mentioned S-polynomials from the \grobner{} basis
computation as long as we guarantee to consider the other two, for more details
on this we refer to~\cite{gmInstallation1988}.

Gebauer and M\"oller present in~\cite{gmInstallation1988} a very efficient
implementation of a combination of the Product and the Chain criterion in a
Buchberger-like algorithm to compute \grobner{} bases.\footnote{Since we assume
that the reader is more familiar with the installation of Gebauer and M\"oller 
than with signature-based \grobner{} basis algorithms we only
present the pseudo code of \rba{} in Algorithm~\ref{alg:rba}.}

In 2002 Faug\`ere published the \ff{} algorithm, the first \grobner{} basis
algorithm using signatures to predict zero reductions. In the last
years many researchers have contributed to optimize and generalize this attempt,
for
example~\cite{huang2010, ggvGGV2010, apF5CritRevised2011, epSig2011, gvwGVW2011,
sw2011b, rs-2012, rs-ext-2012, erF5SB2013, panhuwang2013}.

We can give a general description of signature-based
\grobner{} basis algorithms, denoted \rba{}, similar to Buchberger's algorithm in
Algorithm~\ref{alg:rba}. It depends on the
algorithm~\rewritablenoarg{} which implements the Rewritten
criterion (Lemma~\ref{lem:rew-crit}).

\begin{algorithm}
\begin{algorithmic}[1]
\Require Ideal $I=\langle \gen f \rangle \subset \ring$, monomial
order $\leq$ on $\ring$ and a compatible extension on $\module$, a rewrite order
$\rleq$ on $\basis \cup \hdsyz$
\Ensure Signature \grobner{} basis \basis{} for $I$, \grobner{} basis $\hdsyz$ for $\syz{\gen f}$
\State $\basis{}\gets\emptyset$, $\hdsyz \gets \emptyset$
\State $\pairset\gets\set{\mbasis 1,\ldots,\mbasis m}$
\State $\hdsyz\gets \set{f_i \mbasis j - f_j \mbasis i \mid 1\leq i < j \leq
  m}\subset \module$\label{alg:rba:initialsyz}
\While{$\pairset\neq\emptyset$}
  \State $\beta\gets \text{ element of minimal signature w.r.t. $\leq$ from } \pairset$\label{alg:rba:choosespair}
  \State $\pairset\gets \pairset\setminus\set \beta$
  \If {$\notrewritable\beta$} \label{alg:rba:rewritecheck}
    \State $\gamma\gets$ result of regular \sreducing{}
    $\beta$\label{alg:rba:regular}
    \If {$\proj{\gamma}=0$}
      \State $\hdsyz\gets \hdsyz+\set{\gamma}$\label{alg:rba:addsyz}
    \Else \label{alg:rba:addsingularelement}
      \State $\pairset\gets \pairset\cup\setBuilder
        {\spair\alpha\gamma}
        {\alpha\in\basis{}\text{, $\spair\alpha\gamma$
                                regular}}$\label{alg:rba:pairs}
      \State $\basis{}\gets\basis\cup\set\gamma$\label{alg:rba:basis}
      \State $\updatesyzwitharg$\label{alg:rba:updatesyz}
    \EndIf
  \EndIf
\EndWhile
\State \textbf{return} $(\basis{}, \hdsyz)$
\end{algorithmic}
\caption{\rba{} (Rewrite Basis Algorithm)}
\label{alg:rba}
\end{algorithm}

\begin{algorithm}
\begin{algorithmic}[1]
\Require S-pair $a\alpha - b\beta \in\module$
\Ensure ``true'' if S-pair is rewritable; else ``false''
\If {$a\alpha$ is rewritable}
  \State \textbf{return} true
\EndIf
\If {$b\beta$ is rewritable}
  \State \textbf{return} true
\EndIf
\State \textbf{return} false
\end{algorithmic}
\caption{\rewritablenoarg{} (Rewritten Criterion Check)}
\label{alg:rewritable}
\end{algorithm}

In the following we focus on the rewritability of elements and the generation of
syzygies in order to detect useless zero reductions. We refer
to~\cite{rs-2012,erF5SB2013} for other details
of \rba{}.

Until now $\updatesyz$ in Line~\ref{alg:rba:updatesyz} is just a placeholder
for a generic subalgorithm. In
Section~\ref{sec:improve-rewritten} we make use of $\updatesyz$ in order to
improve the detection of zero reductions.

Next we state the main signature-based criterion to remove useless data.

\begin{lemma}[Rewritten criterion]
For signature $T$ \rba{}
needs to handle exactly one $a\alpha \in \module$ from the set
\begin{equation}
\label{eq:rew-crit}
\mathcal{C}_T = \left\{a\alpha \mid \alpha \in
\basis \cup \hdsyz, a \in \mon \textrm{
and } \sig{a\alpha} = T \right\}.
\end{equation}
\label{lem:rew-crit}
\end{lemma}
\vspace*{-5mm}
Usually the Rewritten criterion as stated above is only defined for $\alpha \in
\basis$. The second criterion is then the so-called \emph{\ff{} criterion} or
\emph{Syzygy criterion}. This criterion states that whenever the signature of an
S-pair is a multiple of the lead term of a syzygy in $\module$, then the
pair can be removed. $\alpha \in \hdsyz$ means that $\alpha$ is a
syzygy and its signature is equal to the lead term in $\module$. Clearly,
whenever in a situation such that $\alpha \in \hdsyz$ and
$a\alpha\in\mathcal{C}_T$ we choose $\alpha \in \hdsyz$ in $\mathcal{C}_T$,
since then we do not need to do any computation in signature $T =
\sig{a\alpha}$.

The choice in Lemma~\ref{lem:rew-crit} depends on a rewrite
order $\rleq$: 

\begin{definition}
A \emph{rewrite order $\rleq$} is a total order on $\basis$ such that
$\sig\alpha \mid \sig\beta \Rightarrow \alpha \rleq \beta$.
\label{def:rew-order}
\end{definition}

Thus it makes sense to choose $\max_\rleq \mathcal{C}_T$ in
Lemma~\ref{lem:rew-crit} and remove all other corresponding S-pairs in signature
$T$ during the computations of \rba{}.

\begin{definition}
\label{def:rewriter}
We call \emph{$\delta:=\max_\rleq \mathcal{C}_T$ the
canonical rewriter in signature $T$ w.r.t. $\rleq$}. If $\sig{a\alpha} = T$ but
$a\alpha$ is not the canonical rewriter in $T$ w.r.t. $\rleq$ then we say that
\emph{$a\alpha$ (or $\alpha$) is rewritable (in signature $T$ w.r.t. $\rleq$)}.
If $\delta \in \hdsyz$ then we say that \emph{$a\alpha$ (or $\alpha$) is
rewritable w.r.t. $\hdsyz$}. Analogously, if $\delta \in \basis$, we use the
notation \emph{rewritable w.r.t. $\basis$}. 
\end{definition}

Two different rewrite orders are used widely these days. For both it holds that
$\alpha \rleq \beta$ for $\alpha \in \basis, \beta \in \hdsyz$. 
This condition defines the usually separately defined Syzygy criterion: Whenever
we have a syzygy $\alpha \in \hdsyz$ such that $\sig{a\alpha} \mid T$ for some
signature $T$ and some multiple $a\in \ring$ then $\alpha$ is the canonical
rewriter. Since we already know that $\proj\alpha = 0$ we do not need to compute
anything for signature $T$ and can go on.
To compare elements from $\basis$ one of the following definitions is used:
\begin{enumerate}
\item $\alpha \rleqff \beta$ if
$\alpha$ has been added to \basis{} before $\beta$ is added to \basis{}. Break
ties arbitrarily.
\item\label{def:specific-rewrite-rules:sb}
$\alpha \rleqsb \beta$ if $\sig\alpha \hdp\beta < \sig\beta
\hdp\alpha$ or if $\sig\alpha \hdp\beta = \sig\beta \hdp\alpha$ and $\sig\alpha
< \sig\beta$.
\end{enumerate}
In the following (besides the examples) we only need the general definition of $\rleq$.
Nevertheless, the reader can think of one of the above mentioned,
particular implementations if this appears to be
helpful.

\begin{remark}
\label{rem:sig-poly-pairs}
Note that \rba{} as presented in Algorithm~\ref{alg:rba} computes
not only the signature \grobner{} basis \basis{} for the input ideal $I$, but
also a \grobner{} basis for the corresponding syzygy module \hdsyz{}, similar
to~\cite{mmtSyzygies1992}.
\end{remark}

Since we are only interested in an efficient
computation of $\basis$ it makes sense to optimize \rba{}:
\sreduction{} and the Rewritten criterion depend solely on the correctness
of the signatures, the lead terms of the module representations in
$\module$. Storing only the signature and not the full module element computations
become more efficient (only reducing polynomial
data, signatures are unchanged once generated), but we loose information stored
in full module representation, for example, to generate more syzygies.

In the following \rba{} is always assumed to use sig-poly pairs only.
We want to recover as many (lead terms of) syzygies as possible without
taking on the burden of computations in $\module$. 

Since we use the names of the above mentioned criteria a lot in the following let us agree on
the following shorthand notations.

\begin{notation}
We denote the Product criterion by \prc{}, the Chain criterion by \chc{} and the
Rewritten criterion by \rwc{}.
\end{notation}

\section{Improving RC}
\label{sec:improve-rewritten}
In this section we want to see how \rwc{} depends on the number
of known syzygies \rba{} can use. We show how one can use
$\updatesyz$ to improve the prediction of zero reductions. For this, let us
first take a look at an example using \rba{}.
\begin{example}
\label{ex:prod-crit}
Let $\field$ be the finite field with $7$ elements and let $\ring =
\field[x,y,z,t]$. Let $<$ be the graded reverse lexicographical monomial order
which we extend to $\schl$ on $\ring^4$. Consider the input ideal $I$ generated
by $f_1=yz-z^2$, $f_2=y^2-xt$, $f_3=xy-xz$ and $f_4=x^2-xy$.
We present
the calculations done by \rba{} using $\rleqsb$ in
Figure~\ref{fig:ex:prod-crit}.
\begin{center}
\begin{figure}
\centering
\setlength{\tabcolsep}{0.8em}
\begin{tabular}{cccc}
$\gbasis i\in\basis$ & reduced from & $\overline{\gbasis i}$ & $\sig{\gbasis i}$ \\
\hline
$\gbasis 1$ & $\mbasis 1$ & $yz-z^2$ & $\mbasis 1$\\
$\gbasis 2$ & $\mbasis 2$ & $y^2-xt$ & $\mbasis 2$\\
$\gbasis 3$ & $\mbasis 3$ & $xy-xz$ & $\mbasis 3$\\
$\gbasis 4$ & $\mbasis 4$ & $x^2-xz$ & $\mbasis 4$\\
$\gbasis 5$ & $\spair{\gbasis 2}{\gbasis 1}=z\gbasis 2 - y \gbasis 1$ &
$z^3-xzt$ & $z \mbasis 2$\\
$\gbasis 6$ & $\spair{\gbasis 3}{\gbasis 2}=y \gbasis 3 - x \gbasis 2$ &
$xz^2-xzt$ & $y \mbasis 3$\\
\end{tabular}
\caption{Computations for \rba{} in
  Example~\ref{ex:prod-crit}.}
\label{fig:ex:prod-crit}
\end{figure}
\end{center}
\rba{} computes $5$ zero reductions corresponding to the following syzygies:
\begin{center}
$
\begin{array}{ll}
\sigma_1 = (y+z) \alpha_4 - (x-y) \alpha_3, &\sigma_2 = (y-z) \alpha_5 - (z^2-xt) \alpha_1,\\
\sigma_3 = (y-z) \alpha_6 - (z^2-zt) \alpha_3, &\sigma_4 = (x-y) \alpha_6 -
(z^2-zt) \alpha_4,
\end{array}
$
$
\begin{array}{c}
\sigma_5 = (x^2-xz) \alpha_5 - (z^3-xzt) \alpha_4.
\end{array}
$
\end{center}
\end{example}
Note that \rba{} using $\potl$ or $\dpotl$ computes the same example
with $4$ zero reductions, respectively.
Even \singular{}'s~\cite{singular400} \grobner{} basis implementation (Gebauer-M\"oller installation)
misses only $4$ zero reductions, so we must be able to improve the prediction
for \rba{} by using $\updatesyz$. The idea is to implement $\updatesyz$ in a way
to recover more known syzygies without blowing up $\hdsyz$ too much by possibly
adding redundant elements:

\begin{characteristic}
\label{char:updatesyz}
Possible implementations of $\updatesyz$ are depending on the module monomial
order $<$:
\begin{enumerate}
\item With $\potl$ \rba{} computes $\basis$ by increasing module indices.
Thus, once an element $\gamma$ of index $k$ is added to $\basis$ such that 
$
k > \max\left\{\ind\alpha \mid \alpha \in
\basis\backslash\{\gamma\}\right\}
$
then all new elements have index $k$ and $\basis$ is a signature
\grobner{} basis for the input up to index $k-1$. So we can add all syzygies
$\proj{\alpha}\gamma - \proj\gamma \alpha,\;\alpha \in
\basis\backslash\{\gamma\}$ resp. signatures $\hdp\alpha \gamma$
to $\hdsyz$. In this way, if the input forms a regular sequence, \rba{} using
$\potl$ does not compute any zero reduction.


\item For $\schl$ and $\dpotl$ we implement $\updatesyz$ such that
whenever an element $\gamma$ is added to $\basis$
we add all syzygies
$\proj{\alpha}\gamma - \proj\gamma \alpha,\;\alpha \in
\basis\backslash\{\gamma\}$ resp. signatures $\hdp\alpha \gamma$ to $\hdsyz$
that increase $L\left(\hdsyz\right)$. Note that in this situation we do not need to take care of the
connection between $\ind\gamma$ and $\ind\alpha$ since this time we might have a
non-incremental computation.
\end{enumerate}
\end{characteristic}

Using $\updatesyz$ as explained in Characteristic~\ref{char:updatesyz}
\rba{} behaves way better in Example~\ref{ex:prod-crit}:
For $\potl$ resp. $\dpotl$ we receive $2$ zero reductions.
Since the input is homogeneous the number of zero reductions for $\potl$ and $\dpotl$
coincides.\footnote{Note that for affine input this need not be the case due
to possible degree drops during the computation of \basis{}.}

For $\schl$ we only drop from $5$ to $3$ syzygies: $\sigma_2$ and $\sigma_3$ are
now detected in $\rewritablenoarg$ due to additional elements in $\hdsyz$.

\begin{remark}
Note that one can optimize $\hdsyz$ even further when more algebraic structure
of the input is known. For example, in~\cite{FSS10b}, bihomogeneous input is
investigated. Using this fact one can construct even more syzygies or relations
by taking information about the corresponding Jacobian matrices into account.
\end{remark}

\section{RC covers CC}
\label{sec:covers-chain-crit}
Looking at equations~\ref{eq:chain-crit} and~\ref{eq:rew-crit} it seems that
there is a strong connection between both criteria. In~\cite{gashPhD2008} Gash
shortly discusses the possibility of adding \chc{} to \ff{} in
special situations. Here we show that \rwc{} completely covers
\chc{} and thus need not be added to \rba{} (or \ff{} as a specific
implementation of it).

\begin{theorem}
Let $\alpha,\beta,\gamma\in\module$ such that
$(\proj\alpha,\proj\beta,\proj\gamma)$ fulfills \chc{}.
Then \rwc{} removes one of the corresponding S-pairs
$\spair\alpha\beta$, $\spair\alpha\gamma$ resp. $\spair\gamma\beta$.
\label{thm:rew-chain-crit}
\end{theorem}

\begin{proof}
Since $\left(\proj\alpha,\proj\beta,\proj\gamma\right)$ fulfills \chc{}
there exist by Equation~\ref{eq:chain-crit} monomials $u,v \in \mon$
such that
\begin{equation}
\proj{\spair\alpha\beta} = u\;\proj{\spair\alpha\gamma} + v\;
\proj{\spair\gamma\beta}.
\label{eq:chain-crit-spair}
\end{equation}
For corresponding monomial multiples this corresponds, on the polynomial side, to
\[a_\beta \proj\alpha - b_\alpha \proj\beta =
u \left(a_\gamma \proj\alpha - c_\alpha \proj\gamma\right) +
v \left(c_\beta \proj\gamma - b_\gamma \proj\beta\right)
\]
such that $a_\beta = u a_\gamma$, $b_\alpha = v b_\gamma$ and
$u c_\alpha = v c_\beta$.
This has the effect that the corresponding multiplied signatures coincide,
too. $\mathcal{T} := \left\{\sig{a_\beta\alpha}, \sig{b_\alpha\beta},
\sig{u c_\alpha \gamma} \right\}$ denotes the set of all appearing signatures. It follows that $\max_< \mathcal{T}$ appears twice in
Equation~\ref{eq:chain-crit-spair}.\footnote{Note that $\sig{uc_\alpha
\gamma} = \max_< \mathcal{T}$ is possible.} Thus, depending on $u$ and $v$ one
of the two S-pairs with maximal signature is detected by \rwc{}
due to the other one.
\end{proof}

Theorem~\ref{thm:rew-chain-crit} states that whenever an S-polynomial would be
discarded by \chc{}, then the corresponding S-pair is removed by \rwc{}.
Thus, a signature-based \grobner{} basis algorithm implementing \rwc{}
as stated in Lemma~\ref{lem:rew-crit} includes \chc{}.

\begin{remark}\
\begin{enumerate}
\item Depending on $u$ and $v$ one has a choice which of the three S-polynomials
fulfilling \chc{} can be removed. In a signature-based
\grobner{} basis algorithm the chosen rewrite order $\rleq$ uniquely
defines the S-pair to be discarded in such a situation.
\item \chc{} corresponds to a found syzygy. The problem for
an efficient implementation of a signature-based \grobner{} basis algorithm
using sig-poly pairs (see Remark~\ref{rem:sig-poly-pairs}) is that one cannot
track this syzygy: One only has the information of the signature stored. In the
setting of Theorem~\ref{thm:rew-chain-crit} two module lead terms cancel out
each other, that means the signatures are the same. Due to the fact that the
tail of the syzygy is not stored the algorithm cannot recompute the new module
lead term resp. signature that could be added to $\hdsyz$. Thus, although we can
remove an S-pair corresponding to this relation in \rba{} we are not able to use
this information on a more global level by adding information to $\hdsyz$.
\end{enumerate}
\end{remark}

As seen in Theorem~\ref{thm:rew-chain-crit} \chc{} is just a particular case of
\rwc{}. Thus it is clear that \rwc{} can remove more elements than \chc{}.

\begin{example}
\label{ex:rew-stronger-than-chain-crit}
Let $\field$ be the finite field with $7$ elements and let $\ring =
\field[x,y,z,t]$. Let $<$ be the graded reverse lexicographical monomial order
which we extend to $\potl$ on $\ring^3$. Consider the input ideal $I$ generated
by $f_1 = yz-2t^2$, $f_2 = xy+t^2$, and $f_3 = x^2z+3xt^2-2yt^2$. We present
the calculations done by \rba{} using $\rleqsb$ in
Figure~\ref{fig:ex:rew-stronger-than-chain-crit}.

\rba{} removes $\spair{\gbasis 6}{\gbasis 1} = y \gbasis 6 - z^2t^2 \gbasis
1$ due to the rewriting of $\gbasis 5$:
$\sig{y\gbasis 6} = \sig{t^2 \gbasis 5}$ and $\hdp{y\gbasis 6} = yz^3t^2 >
y^2t^4 = \hdp{t^2 \gbasis 5}$. 
Recalculating this rewriting one gets
\begin{align*}
\spair{\gbasis 6}{\gbasis 1} &= y \gbasis 6 - z^2t^2 \gbasis 1 = t^2\left(\gbasis 5 + (2y+3z) \gbasis 3 - t^2 \gbasis 1\right)
\end{align*}
where $\hdp{\gbasis 5} > \hdp{y\gbasis 3} > \hdp{t^2 \gbasis 1}$. \chc{} does
not detect this relation, even not in an optimized Gebauer-M\"oller
variant as implemented in \singular{} (\cite{singular400}). There, a reduction
to zero is computed.
\begin{center}
\begin{figure}
\centering
\setlength{\tabcolsep}{0.8em}
\begin{tabular}{cccc}
$\gbasis i\in\basis$ & reduced from & $\hd{\overline{\gbasis i}}$ & $\sig{\gbasis i}$ \\
\hline
$\gbasis 1$ & $\mbasis 1$ & $yz$ & $\mbasis 1$\\
$\gbasis 2$ & $\mbasis 2$ & $xy$ & $\mbasis 2$\\
$\gbasis 3$ & $\spair{\gbasis 2}{\gbasis 1}=z\gbasis 2 - x \gbasis 1$ &
$xt^2$ & $z \mbasis 2$\\
$\gbasis 4$ & $\mbasis 3$ & $x^2z$ & $\mbasis 3$\\
$\gbasis 5$ & $\spair{\gbasis 4}{\gbasis 2}=y \gbasis 4 - xz \gbasis 2$ &
$y^2t^2$ & $y \mbasis 3$\\
$\gbasis 6$ & $\spair{\gbasis 4}{\gbasis 3}=t^2 \gbasis 4 - xz \gbasis 3$ &
$z^3t^2$ & $t^2 \mbasis 3$\\
\end{tabular}
\caption{Computations for \rba{} in
  Example~\ref{ex:rew-stronger-than-chain-crit}.}
\label{fig:ex:rew-stronger-than-chain-crit}
\end{figure}
\end{center}
\end{example}

\begin{remark}
Any signature-based \grobner{} basis algorithm implementing \rwc{} as
discussed in sections~\ref{sec:detection} and~\ref{sec:improve-rewritten} does
not need any further modifications to detect the useless computations
\chc{} predict. This leads to an easier description of the algorithm compared to
Theorem~3.1 in~\cite{gerdtHashemiG2V} where \chc{} is explicitly added as
conditions $C_2$ and $C_3$. There Gerdt and Hashemi add a
Gebauer-M\"oller-like criteria check to their modified \gggv{} algorithm.
\end{remark}

\section{Does RC cover PC, too?}
\label{sec:covers-prod-crit}
As for \chc{} we wish to find a connection between \rwc{} and \prc{}.
The nice fact is that we can easily translate Lemma~\ref{lem:prod-crit} to
the signature-based world.
Note the subtle difference that the argument of Corollary~\ref{cor:sig-prod-crit}
holds for regular \sreductions{} whereas Lemma~\ref{lem:prod-crit} does not have
this restriction.

\begin{corollary}[Variant of Lemma~\ref{lem:prod-crit}] \
Let $\alpha,\beta \in \module$ with $\lcmm\alpha\beta = \hdp\alpha
\hdp\beta$. Then $\spair\alpha\beta$ regular \sreduces{} to zero w.r.t.
$\left\{\alpha,\beta\right\}$.
\label{cor:sig-prod-crit}
\end{corollary}

\begin{proof}
Let $T = \max_<\left\{\sig{\hdp\beta \alpha},\sig{\hdp\alpha \beta}\right\}$.
For any $u \in \support{\proj\beta - \hdp\beta}$,
$v \in \support{\proj\alpha - \hdp\alpha}$
it holds that $\sig{u\alpha} <T$ and $\sig{v\beta}<T$.
\end{proof}

The outcome of Corollary~\ref{cor:sig-prod-crit} is that in \rba{}
\prc{} can be used without any further restrictions or modifications. This is
something already discussed
in~\cite{gashPhD2008,ghmInvolutiveF5-2013,gerdtHashemiG2V}. With this it is
clear that one can modify any signature-based \grobner{} basis algorithm by
adding one of the following steps:
\begin{enumerate}
\item Check \prc{} explicitly as it is done, for example, in
\gggv{} in~\cite{gerdtHashemiG2V} (condition $C_1$).
\item Add all possible principal syzygies to \hdsyz{} whenever a new element
$\gamma$ with $\proj\gamma \neq 0$ is added to \basis{}. This is done, for
example, in the 2013 revision of the \gvw{} algorithm, see Step $4b(b1)$ in
Figure~3.1 in~\cite{gvwGVW2013}.
\end{enumerate}

Clearly, both of these possible optimizations add an overhead to the algorithm,
for example in the second case the signatures of most of the added principal
syzygies are multiples of signatures already available in \hdsyz{}. Thus one
also has to interreduce \hdsyz{} in order to have efficient checks of
\rwc{} in the following. 
In other words, it makes sense to ask the following more algebraic question:
Does \rwc{} cover \prc{}? If not, which ones are not covered, and how can one
handle thses in \rba{} efficiently?

Finding answers to these questions seems to be an easy task:
\prc{} is based on the
fact that $\proj\beta \alpha - \proj\alpha \beta \in \module$ is a syzygy. So
the only question is to see if \rba{} finds the corresponding signature 
\[
\max_<\left\{\sig{\hdp\beta \alpha},\sig{\hdp\alpha \beta}\right\}
=\max_<\left\{\hdp\beta \sig\alpha,\hdp\alpha \sig\beta\right\}.
\]

The answer to this question is not trivial as it seems to depend in
the signature-based world on the chosen module monomial order.
Almost always \prc{} is covered by \rwc{}, still one can construct
counterexamples:

\begin{example}[Example~\ref{ex:prod-crit} continued]
In Section~\ref{sec:improve-rewritten} we have implemented $\updatesyz$ in order to
strengthen \rwc{} by relying on more known syzygies. Still, for
$\schl$ as module monomial order we still had $3$ zero reductions not discarded
in advance: 
$\sigma_1 = (y+z) \alpha_4 - (x-y) \alpha_3$,
$\sigma_4 = (x-y) \alpha_6 - (z^2-zt) \alpha_4$ and
$\sigma_5 = (x^2-xz) \alpha_5 - (z^3-xzt) \alpha_4$.
We see that $\sigma_5$ can be presented as
$\sigma_5 = \proj{\alpha_4} \alpha_5 - \proj{\alpha_5} \alpha_4$.
This means that $\sigma_3$ corresponds to a syzygy coming from an S-pair
fulfilling \prc{}, $\spair{\alpha_5}{\alpha_4}$, which is not
detected by \rwc{} in \rba{}.
\end{example}

Note that \rba{} using $\potl$ or $\dpotl$ does remove all S-pairs fulfilling
\prc{}. Moreover, if one slightly changes $\schl$ to use the
variant of $\potl$ that prefers the smaller indices (see note after
Definition~\ref{def:module-monomial-order}) then \rwc{} also covers
\prc{}. However, 
mirroring this change in the input binomials one can easily construct another
system where \rba{} does not remove all S-pairs fulfilling \prc{} with this
module monomial order.

Thus we want to use the fact that not all syzygies coming from S-pairs
fulfilling \prc{} are found in \rba{}.
Due to Corollary~\ref{cor:sig-prod-crit} we
can easily implement it in \rba{} without affecting correctness or termination.
In order to do this efficiently, one should take a bit care: Since
\rba{} is a signature-based \grobner{} basis algorithm its criterion,
  \rwc{}, should be favored over \prc{}. Thus we might
change Algorithm~\ref{alg:rewritable} to the variant presented in
Algorithm~\ref{alg:rewritable-pc}.

\begin{algorithm}
\begin{algorithmic}[1]
\Require S-pair $a\alpha - b\beta \in\module$
\Ensure ``true'' if S-pair is rewritable or fulfills the Product criterion; else ``false''
\If {$a\alpha$ or $b\beta$ is rewritable w.r.t.
  $\hdsyz$}\label{alg:rewritable-pc:syz}
  \State \textbf{return} true
\EndIf
\If {$a\proj\alpha-b\proj\beta$ fulfills the Product
  criterion}\label{alg:rewritable-pc:pc}
  \State $\hdsyz \gets \hdsyz \cup \{\proj\alpha \beta - \proj\beta\alpha\}$
  \label{alg:rewritable-pc:add-syz}
  \State \textbf{return} true
\EndIf
\If {$a\alpha$ or $b\beta$ is rewritable w.r.t.
  $\basis$}\label{alg:rewritable-pc:rew}
  \State \textbf{return} true
\EndIf
\State \textbf{return} false
\end{algorithmic}
\caption{\rewritablenoargnew{} (Rewritten \& Product Criterion Check)}
\label{alg:rewritable-pc}
\end{algorithm}

Comparing $\rewritablenoargnew$ with $\rewritablenoarg$ we see first that
\rwc{} is split up: In Line~\ref{alg:rewritable-pc:syz} we check
with elements in $\hdsyz$, in Line~\ref{alg:rewritable-pc:rew} we check with
elements in $\basis$. Between those two tests, we check for \prc{}
(Line~\ref{alg:rewritable-pc:pc}). In this way the overhead is
minimized: 
\begin{enumerate}
\item If an S-pair is already removed in Line~\ref{alg:rewritable-pc:syz} then
we have already a corresponding syzygy. This happens in
Example~\ref{ex:prod-crit} for $5$ S-pairs fulfilling \prc{}.
\item If an S-pair was not removed in the first step, but fulfills
\prc{} then the corresponding syzygy $\sigma = \proj\alpha \beta - \proj\beta \alpha$
is missing in $\hdsyz$. We can remove $\spair\alpha\beta$ due to
Corollary~\ref{cor:sig-prod-crit}. Here it makes sense to add $\sigma$ to
$\hdsyz$ since there exists no $\sigma' \in \hdsyz \backslash\{\sigma\}$ such
that $\sig{\sigma'} \mid \sig{\sigma}$.
\item If $\spair\alpha\beta$ was not removed in the first two steps, we check
for rewriters in $\basis$.
\end{enumerate}

Above we said that \rba{} shall favor \rwc{} over \prc{}, so why do we check
rewritability w.r.t. $\basis$ last? The answer
to this question is that being rewritable w.r.t. $\basis$ is a local property.
At some point \rba{} might have computed enough successors of elements in
$\basis$ that a canonical rewriter in signature $T$ might no longer be the
canonical one in signature $tT$ for some $t\in\mon$. Thus the relation stored in
this canonical rewriter is no longer available to \rba{}. Having instead a new syzygy
$\sigma \in \hdsyz$ this is a global canonical rewriter that removes useless
S-pairs in any signature that is a multiple of $\sig\sigma$. So even if for
$a\proj\alpha - b\proj\beta$ fulfilling \prc{} $a\alpha$ or
$b\beta$ is rewritable w.r.t. $\basis$ it makes sense to check \prc{} first
and add a new syzygy to $\hdsyz$ as we see in
Section~\ref{sec:exp-results}.

Looking again at Example~\ref{ex:prod-crit} we see that a variant of \rba{} using
$\rewritablenoargnew$ detects $\spair{\alpha_5}{\alpha_4}$ and adds the
corresponding syzygy $\sigma_5$ without computing a zero reduction. Thus also
when using $\schl$ \rba{} predicts all except $2$ zero reductions.

As a last fact let us compare the ideas behind Gebauer and M\"oller's
implementation of \prc{} and \chc{} in terms of efficiency to
signature-based \grobner{} basis algorithms:
Assume that $\left(\proj\alpha,\proj\beta,\proj\gamma\right)$ fulfills
\chc{} such that
$
\lcm\left(\hdp \alpha,\hdp\beta\right) = \lcm\left(\hdp \alpha,\hdp\gamma\right).
$
Now both, $\proj{\spair\alpha\beta}$
and $\proj{\spair\alpha\gamma}$ can be removed. The hard part is to not remove both at the
same time. Gebauer and M\"oller implemented a step-by-step check of
\chc{} and \prc{}, in order to remove useless S-polynomials as early as possible
including a check to not remove both S-poly\-nomials in the above situation.
Looking at this from the signature-based point of view, $\sig{\spair\alpha\beta}
= \sig{\spair\alpha\gamma}$. Due to the rewrite order and the handling of
S-pairs by increasing signature in \rba{} one of the two S-pairs is handled
first, say $\spair\alpha\beta$. This means that (if there is no other criterion
to remove it) $\spair\alpha\beta$ is further reduced whereas
$\spair\alpha\gamma$ is rewritten. Thus it is not possible to remove too many
elements in such a chain when using signature-based \grobner{} basis algorithms.

Moreover, the Gebauer and M\"oller implementation checks \prc{}
last in order to remove correctly as much as possible useless S-polynomials.
For example, for $\left(\proj\alpha,\proj\beta,\proj\gamma\right)$ fulfilling  
\[
\proj{\spair\alpha\beta} = u\proj{\spair\alpha\gamma} + v\;
\proj{\spair\gamma\beta}
\]
it is possible that one removes $\proj{\spair\alpha\beta}$. Later on,
$\proj{\spair\alpha\gamma}$ fulfills \prc{} and is removed, too. In this
situation, out of three S-polynomials only one, $\proj{\spair\beta\gamma}$ needs
to be further reduced. In a signature-based algorithm the
question whether \rwc{} w.r.t. $\hdsyz$ detects one of these elements naturally comes up.
If $\sig{\spair\alpha\gamma}$ is rewritable w.r.t. $\hdsyz$ then
$\sig{\spair\alpha\beta}$ is so, too. Thus also in
\rba{} only
one S-pair would be left.

\section{Covering PC in RBA using $\potl$}
\label{sec:potl-covers-prod-crit}
An astonishing point is the fact that when \rba{} uses $\potl$ we have not
found any example where an S-pair fulfilling \prc{} is not
already rewritable w.r.t. $\hdsyz$. Moreover, this behaviour seems to not depend
on the ideal at all since we tried millions of examples from homogeneous to
affine, from zero to higher dimensional. In all cases it holds that when using
$\rewritablenoargnew$ in \rba{} \prc{} was not used once to remove an S-pair.
Let us take a more detailed look at this situation:

$\potl$ enforces \rba{} to compute incrementally:
For each $i$ a signature \grobner{} basis $\basis_i$ for
$\langle f_1,\ldots,f_{i}\rangle$ is
computed. Afterwards $f_{i+1}$ enters the computations and new S-pairs are
handled until a signature \grobner{} basis for $\langle
f_1,\ldots,f_{i+1}\rangle$ is achieved. With the ideas of~\cite{epF5C2009}
we can assume to have the reduced \grobner{} basis $\proj{\basis_i}=\{\proj{\mbasis
  1},\ldots,\proj{\mbasis {k-1}}\} \subset \ring$ for $\langle
f_1,\ldots,f_i\rangle$ before adding $f_{i+1}$ to the computations. Assuming
further to compute in the next incremental step a signature \grobner{} basis for
$\langle \proj{\mbasis 1},\ldots,\proj{\mbasis {k-1}},f_{i+1}\rangle$, we can
adjust notations by setting $f_{i+1} = \proj{\mbasis k}$.
Due to the fact that $\proj{\basis_i}$ is the reduced \grobner{} basis for $\langle
f_1,\ldots,f_i\rangle$ we have more structure
to exploit. On the other hand the incremental run of \rba{} itself may put a penalty on
the efficiency of the computations (often \rba{} performs better using $\schl$
instead of $\potl$).

Next we assume the start at the above incremental step with $\mbasis k$.
We look at S-pairs in the order they are generated by \rba{}: The first possibility
is to built S-pairs between elements of index $k$ and those of index $<k$
(corresponding to polynomials in $\proj{\basis_i}$).

\begin{lemma}
Assume $\potl$. Let $\alpha,\beta \in \module$ such that $\ind\alpha <
\ind\beta$ and $\proj{\spair\alpha\beta}$ fulfills \prc{}.
Then $a\alpha$ or $b\beta$ is rewritable w.r.t. $\hdsyz$.
\label{lem:covers-prod-crit-lower-idx}
\end{lemma}

\begin{proof}
$\sig{\spair\alpha\beta}$ is $\hdp\alpha \sig\beta$ since $\ind\beta
> \ind\alpha$ and we assume $\potl$. By construction $\hdp\alpha \in
L\left(\proj{\basis_i}\right)$ and $\sig\beta$ is a multiple of $\mbasis k$.
Due to Line~\ref{alg:rba:initialsyz} for Algorithm~\ref{alg:rba} there exists
an element $\sigma \in \hdsyz$ such that $\sig\sigma \mid \hdp\alpha\sig\beta$.
\end{proof}

When generating S-pairs with both generators of index $k$ two different
situations can appear: If one of the generators is $\mbasis k$ we can prove the
following statement.

\begin{lemma}
Assume $\potl$. Let $\mbasis k,\beta \in \module$ such that $\ind\beta = k$ and 
$\proj{\spair{\mbasis k}{\beta}}$ fulfills \prc{}. Then $a\alpha$ or
$b\beta$ is rewritable w.r.t. $\hdsyz$.
\label{lem:covers-prod-crit-same-idx-initial}
\end{lemma}

\begin{proof}
Since $\ind\beta = k$ and $\mbasis k \neq \beta$, $\sig\beta = \lambda \mbasis k$
where $\lambda >1$. $\beta$ is a successor of some initial S-pair $\spair{\mbasis
k}{\mbasis j}$ for $j<k$ (possibly over several steps).
The signature of this initial S-pair is $\frac{\lcmm{\mbasis
  k}{\mbasis j}}{\hdp{\mbasis k}} \mbasis k$. Thus $\lambda = \lambda' \frac{\lcmm{\mbasis
  k}{\mbasis j}}{\hdp{\mbasis k}}$ for some $\lambda' \geq 1$.
  Furthermore, generating the S-pair $\spair{\mbasis k}\beta$ we get
\begin{align*}
\hdp{\mbasis k} \sig\beta &= \hdp{\mbasis k} \lambda' \frac{\lcmm{\mbasis
  k}{\mbasis j}}{\hdp{\mbasis k}}  \mbasis k \\
  & = \lambda' \lcmm{\mbasis k}{\mbasis j} \mbasis k 
  = \lambda' \lambda'' \hdp{\mbasis j} \mbasis k.
\end{align*}
By the same argument as in the proof of
Lemma~\ref{lem:covers-prod-crit-lower-idx} there exists a $\sigma \in \hdsyz$
that is the canonical rewriter in signature $\hdp{\mbasis k} \sig\beta$.
\end{proof}

Sadly, the generalization of Lemma~\ref{lem:covers-prod-crit-same-idx-initial}
by replacing $\mbasis k$ with an arbitrary $\alpha$, $\ind\alpha = k$ remains
unproven:

\begin{conjecture}
Assume $\potl$. Let $\alpha,\beta \in \module$ such that $\sig\alpha =
\lambda_\alpha \mbasis k$, $\sig\beta = \lambda_\beta \mbasis k$ for
$\lambda_\alpha,\lambda_\beta > 1$, and 
$\proj{\spair\alpha\beta}$ fulfills \prc{}.
Then $a\alpha$ or $b\beta$ is
rewritable w.r.t. $\hdsyz$.
\label{conj:covers-prod-crit-same-idx}
\end{conjecture}

The difference between Conjecture~\ref{conj:covers-prod-crit-same-idx} and
Lemma~\ref{lem:covers-prod-crit-same-idx-initial} lies in the fact that we loose the connection
between $\hdp\alpha$ and $\hdp{\mbasis k}$. The main gap in the proof of the
above conjecture is the following:
Since we assume $\potl$ all elements $\sigma \in \hdsyz$ with
$\sig\sigma = \lambda \mbasis k$ have $\lambda \in L\left(\proj{\basis_i}\right)$. Moreover
the following holds:
\begin{enumerate}
\item $\hdp\alpha \notin L\left(\proj{\basis_i}\right)$ since otherwise $\alpha$
would have been further \sreduced{} (all reductions with lower index elements
are regular \sreductions{}).
\item For $\sig\beta = \lambda_\beta \mbasis k$, $\lambda_\beta \notin
L\left(\proj{\basis_i}\right)$. Otherwise $\beta$ would not exist in $\basis$
since the S-pair it is reduced from would have been removed by \rwc{} w.r.t.
$\hdsyz$.
\end{enumerate}
Still, if the conjecture is true, it must hold that $\hdp\alpha \lambda_\beta
\in L\left(\proj{\basis_i}\right)$. In contrast to the proof of
Lemma~\ref{lem:covers-prod-crit-same-idx-initial} it is not even clear which
element from $\proj{\basis_i}$ might have constructed the corresponding syzygy.
There are examples where the canonical rewriter $\sigma \in \hdsyz$ could have
$\sig\sigma = \hdp{\mbasis j} \mbasis k$ whereas
$\hdp{\mbasis j}$ is not involved in any predecessor of $\alpha$ or $\beta$ at
all. 

\begin{remark}\
\begin{enumerate}
\item Note that for regular input sequences
Conjecture~\ref{conj:covers-prod-crit-same-idx} is clearly true since
\rba{} detects all syzygies (see, for example, Corollary~3
    in~\cite{fF52002Corrected}).
Moreover, there is a connection between
Conjecture~\ref{conj:covers-prod-crit-same-idx} and the Moreno-Soc\'ias conjecture
(\cite{moreno-socias-1991}) (paper in preparation).
Solving this problem might have a big impact on
\grobner{} basis computations due to revealing algebraic information not used
until now.
\item Clearly one can always add the corresponding principal syzygy as it is
done, for example, in the new version of \gvw{} (see Step~$4b~(b1)$ of
    Figure~3.1 in~\cite{gvwGVW2013}). Still the question of the conjecture is
open: Can the underlying structure of a general signature-based
\grobner{} basis algorithm using $\potl$ already predict those zero reductions
without further modifications? This is not a question of efficient
implementations, but focuses on the algebraic structures hidden underneath. 

Furthermore, do we get any more syzygies resp. relations when adding all
principal syzygies and interreducing \hdsyz{} or are those relations already
covered by an easy implementation of Algorithm~\ref{alg:rewritable-pc}? These
are relavant questions for understanding \grobner{} basis computations.
\end{enumerate}
\end{remark}

\section{Experimental results}
\label{sec:exp-results}
We give some experimental results, all computed over a field of characteristic
$32003$ with the graded reverse lexicographical monomial order $<$. All
computations where done with an implementation of \rba{} in \singular{} (available
since version $4.0.0$). All examples are available under
\begin{center}
\url{https://github.com/ederc/singular-benchmarks}.
\end{center}

Figure~\ref{fig:exp-results} shows the number of zero reductions for the
computation of the corresponding \grobner{} bases. \std{} denotes the
implementation of the Gebauer-M\"oller installation in \singular{}, ``U'' denotes the usage of
$\updatesyz$ as explained in Characteristic~\ref{char:updatesyz}, ``PC'' means
that \rba{} uses $\rewritablenoargnew$. For columns including ``PC'' the number
of Product criteria is given in brackets: those not found when checking
rewritability w.r.t. $\hdsyz$ first, those not found by checking rewritability
w.r.t. $\hdsyz$ and $\basis$ last. Due to the discussion in
Section~\ref{sec:potl-covers-prod-crit} there is no difference between using
$\rewritablenoarg$ or $\rewritablenoargnew$ for \rba{} with $\potl$: In all
examples all syzygies coming from \prc{} are rewitable w.r.t. $\hdsyz$.
\begin{center}
\begin{figure}
\centering
\begin{tabular} {c|c|c|c:c|c:c}
  \multirow{2}{*}{\bf Benchmark} & {\std{}} & {\rba{} $\potl$} &
   \multicolumn{2}{c|}{\rba{} $\schl$} & \multicolumn{2}{c}{\rba{} $\dpotl$}\\ 
  & & \multicolumn{1}{c|}{U}
  &\multicolumn{1}{c:}{U}
  & \multicolumn{1}{c|}{U+PC}
  &\multicolumn{1}{c:}{U}
  & \multicolumn{1}{c}{U+PC}\\
\hline
cyclic-8 & 4284 & 243 & 771 & 771(17,0) & 243 & 243(7,0)\\ 
cyclic-8-h & 5843 & 243 & 771 & 771(17,0) & 243 & 243(7,0)\\ 
eco-11 & 3476 & 0 & 614 & 614(770,0) & 541 & 538(556,0)\\ 
eco-11-h & 5429 & 502 & 629 & 608(57,0) & 502 & 502(10,0)\\ 
f-744 & 589 & 0 & 248 & 244(99,0) & 185 & 184(61,0)\\ 
f-744-h & 1267 & 189 & 248 & 244(49,0) & 189 & 189(42,0)\\ 
katsura-11 & 3933 & 0 & 348 & 304(275,0) & 0 & 0(62,0)\\ 
katsura-11-h & 3933 & 0 & 348 & 304(275,0) & 0 & 0(62,0)\\ 
noon-9 & 25508 & 0 & 682 & 646(505,0) & 0 & 0(21,0)\\ 
noon-9-h & 25508 & 0 & 682 & 646(505,0) & 0 & 0(21,0)\\ 
binomial-6-2 & 21 & 6 & 15 & 8(16,7) & 6 & 6(11,0)\\ 
binomial-6-3 & 20 & 13 & 15 & 9(6,6) & 13 & 13(4,0)\\ 
binomial-7-3 & 27 & 24 & 21 & 21(9,0) & 24 & 24(6,0)\\ 
binomial-7-4 & 41 & 16 & 19 & 16(8,3) & 16 & 16(5,0)\\ 
binomial-8-3 & 53 & 23 & 27 & 27(10,0) & 23 & 23(0,0)\\ 
binomial-8-4 & 40 & 31 & 26 & 26(3,0) & 16 & 31(0,0)\\ 
\hline
\end{tabular}
\caption{\# zero reductions and not detected Product criteria}
\label{fig:exp-results}
\end{figure}
\end{center}
\rwc{} almost always covers \prc{}
completey, but often first by using the rewritability check w.r.t. $\basis$.
This means that adding the signature of $\spair\alpha\beta$ in such a situation
enlarges $\hdsyz$ and thus might strengthen \rwc{}. Having
tested tens of millions of examples until now it is very rarely the case that
an S-pair fulfilling \prc{} is not detected by \rwc{} at all. Random systems
behave like the benchmarks given in Figure~\ref{fig:exp-results}, thus we
concentrated on the easiest cases in which such situations appear: binomial
ideals. For example, using $\schl$ for
\texttt{binomial-7-4}\footnote{Notation \texttt{binomial-7-4} means $7$ binomial generators
in a ring of $7$ variables, all homogeneous and of degree $4$.} $8$ S-pairs
fulfilling \prc{}  are not rewritable w.r.t. $\hdsyz$, but only $3$ of
them are not rewritable w.r.t. $\basis$. Still it makes sense to add all the
corresponding syzygies, since this decreases the number of zero reductions ($19$
to $16$). This shows again the difference between the ``local'' rewritable
w.r.t. $\basis$ and the ``global'' rewritable w.r.t. $\hdsyz$. On the other
hand, even if all those S-pairs are detected by the Rewritten criterion it is
sometimes useful to add the syzygies which are not in $\hdsyz$ in order to
decrease the number of zero reductions computed, see, for example,
\texttt{katsura-11} for $\schl$ or affine systems like \texttt{eco-11} or
\texttt{f-744} for $\dpotl$. Note that for homogeneous input
\rba{} using $\dpotl$ might rewrite S-pairs fulfilling \prc{}
only w.r.t. $\basis$, but not w.r.t. $\hdsyz$: Due to reasons of efficiency a
practical implementation of \rba{} initially checks the S-pairs at their
generation directly. At this point \rba{} might not have completed the
computation of all lower degree signature \grobner{} bases.
When using $\rewritablenoargnew$ as stated in
Algorithm~\ref{alg:rba}, that means first checking the S-pair before its potential
\sreduction{} (Line~\ref{alg:rba:rewritecheck}), a \grobner{} basis up to a
certain degree is already computed. So \rba{} with $\dpotl$ behaves in
the same way as \rba{} with $\potl$. Furthermore, note that in any
case the signature-based attempt predicts
many more zero reductions than the Gebauer-M\"oller installation.
\section{Conclusion}
We have given a discussion on different attempts to predict zero reductions
during \grobner{} basis computations. We have seen that \chc{} is
covered by \rwc{}. Still, efficient implementations of
\rba{} storing only the leading term of the syzygy cannot recover the syzygy of
the corresponding chain. So the global influence of the syzygy is hidden behind the
locality of rewriting w.r.t. some element in $\basis$.
Further research in this direction is crucial.

Even though we have shown examples where \prc{} is not completely
covered by \rwc{}. These cases are very rare. We presented an easy way
to make use of \prc{} in \rba{} without introducing overhead. For $\potl$ it seems
that all syzygies coming from \prc{} are already in $\hdsyz$. Proving this
conjecture might give further insight in relations constructed during
\grobner{} basis computations.

Next we investigate the connection to the ideas for generating minimal
sets of critical pairs presented in~\cite{ckr-2004}.
Another attempt is the
generalization of the conecpt of signatures, storing a few more
module terms to recover more syzygies, but keeping the introduced
computational overhead at a minimum. The discussion presented here generalized
different attempts in this direction (\cite{gashPhD2008,gerdtHashemiG2V,gvwGVW2013})
and answered open questions. The intention of this paper is not to introduce a
new, more efficient variant of signature-based algorithms, but to analyze and
to compare the known strategies.
Different views on the prediction of zero reductions may enable
us to exploit more of the algebraic structures behind \grobner{}
basis computations.

\end{document}